\documentclass[12pt]{amsart}
\usepackage[english]{babel}
\usepackage{pgf}
\usepackage[utf8]{inputenc}

\usepackage[centering]{geometry}
\usepackage{mathtools} 
\usepackage[normalem]{ulem} 

\usepackage{array}
\usepackage{tabularx}
\usepackage{graphicx}
\usepackage{amsmath}
\usepackage{amssymb}
\usepackage{amsthm}

\usepackage{enumerate}

\usepackage{tikz}
\usetikzlibrary{arrows}
\usepackage{tkz-euclide}
\usepackage{pgfplots}

\usepackage{color}
\usepackage{xcolor}
\definecolor{myred}{rgb}{0.2,0,0}
\definecolor{myblue}{rgb}{0,0,0.6}
\definecolor{mygreen}{rgb}{0,0.2,0}
\usepackage[colorlinks=true,linkcolor=myred,citecolor=myblue,urlcolor=mygreen]{hyperref}

\numberwithin{equation}{section}

\allowdisplaybreaks

\newcommand{\N}{\mathbb{N}}
\newcommand{\Z}{\mathbb{Z}}
\newcommand{\R}{\mathbb{R}}

\newcommand{\bfw}{\mathbf{w}}

\newcommand{\smallspace}{\hspace{1pt}}
\newcommand{\eqdef}{\coloneqq}

\newcommand{\A}{\mathcal{A}}

\DeclareMathOperator{\e}{\mathrm{e}}

\DeclareMathOperator{\Ker}{\mathrm{Ker}}

\newcommand{\floor}[1]{\left\lfloor #1 \right\rfloor}
\newcommand{\cb}[1]{\left\{#1 \right \}}

\newcommand{\abs}[1]{\left| #1 \right|}
\newcommand{\norm}[1]{\left\| #1 \right\|}

\newcommand{\rb}[1]{\left( #1 \right)}

\newcommand{\ka}{k_1}
\newcommand{\kb}{k_2}
\newcommand{\kc}{k_3}

\newtheorem{theorem}{Theorem}[section]
\newtheorem{lemma}[theorem]{Lemma}
\newtheorem{corollary}[theorem]{Corollary}
\newtheorem{proposition}[theorem]{Proposition}

\theoremstyle{definition}
\newtheorem{definition}[theorem]{Definition}

\theoremstyle{remark}
\newtheorem*{remark}{Remark}

\usepackage{enumitem}

\begin{document}
 \selectlanguage{english}
  \title{Synchronizing automatic sequences along Piatetski-Shapiro sequences}

  \author[J.-M. Deshouillers]{Jean-Marc Deshouillers}
  \address{Institut de Math\'{e}matiques de Bordeaux UMR 5251, Universit\'{e} de Bordeaux 	351, cours de la Lib\'{e}ration -- F 33 405 Talence, France}\email{\href{mailto:jean-marc.deshouillers@math.u-bordeaux.fr}{jean-marc.deshouillers@math.u-bordeaux.fr}}

  \author[M. Drmota]{Michael Drmota}
  \address{Institute of Discrete Mathematics and Geometry, TU Wien, Wiedner Hauptstr. 8--10, A-1040 Wien, Austria}\email{\href{mailto:michael.drmota@tuwien.ac.at}{michael.drmota@tuwien.ac.at}}

  \author[C. M\"ullner]{Clemens M\"ullner}
  \address{Institute for discrete mathematics and geometry, TU Wien, Wiedner Hauptstr. 8–-10, 1040 Wien, Austria}\email{\href{mailto:clemens.muellner@tuwien.ac.at}{clemens.muellner@tuwien.ac.at}}

  \author[A. Shubin]{Andrei Shubin}
  \address{Institute of Discrete Mathematics and Geometry, TU Wien, Wiedner Hauptstr. 8--10, A-1040 Wien, Austria}\email{\href{mailto:andrei.shubin@tuwien.ac.at}{andrei.shubin@tuwien.ac.at}}

  \author[L. Spiegelhofer]{Lukas Spiegelhofer}
  \address{Department Mathematics and Information Technology, Montan\-universit\"at Leoben, Franz-Josef-Strasse 18, 8700 Leoben, Austria}\email{\href{mailto:lukas.spiegelhofer@unileoben.ac.at}{lukas.spiegelhofer@unileoben.ac.at}}


\maketitle

\begin{abstract} 
The purpose of this paper is to study subsequences of synchronizing $k$-automatic sequences $a(n)$ along Piatetski-Shapiro sequences $\lfloor n^c \rfloor$ with non-integer $c>1$. In particular, we show that $a(\lfloor n^c \rfloor)$ satisfies a prime number theorem of the form $\sum_{n\le x} \Lambda(n)a(\lfloor n^c \rfloor) \sim C\, x$, and, furthermore, that it is deterministic for $c \in \mathbb R\setminus \mathbb Z$. As an interesting additional result, we show that the sequence $\lfloor n^c\rfloor \bmod m$ has polynomial subword complexity.
\end{abstract}

\section{Introduction}
\label{sec_Intro}

Suppose that $a(n)$ is a sequence taking its values in a finite alphabet $\mathcal{A}$.
The subword complexity $N_H(a)$ is the
number of different tuples from $\mathcal A^H$ occurring in $a$ as contigous finite subsequences:
\[
N_H(a) \eqdef \bigl\lvert\bigl\{(a(n), a(n+1), \ldots a(n+H-1)) : n\ge 0 \bigr\}\bigr\rvert.
\]
If the subword complexity is sub-exponential, that is, 
\[
\lim_{H\to\infty}  \frac{\log N_H(a)}{H} = 0,
\]
then the sequence $a(n)$ is called \emph{deterministic}.\footnote{More generally, a complex valued sequence $a(n)$
is called deterministic if the topological entropy $h$ of the corresponding dynamical system is zero.}

Sarnak's conjecture \cite{Sarnak2011} says that every deterministic complex valued sequence $a(n)$ 
is orthogonal to the M\"obius function $\mu(n)$, that is, 
\begin{align}\label{eqmoebius}
\sum_{n\le N} \mu(n)\smallspace a(n) = o(N)  \qquad (n\to\infty).
\end{align}

This conjecture is open, however, it could
be proved for several special classes of sequences \cite{Bourgain2013a, Deshouillers2015, Downarowicz2015a, ElAbdalaoui2016, ElAbdalaoui2014, Ferenczi2016, Green2012a, Green2012, Kulaga-Przymus2015, Liu2015, Mauduit2010, Mauduit2015, Muellner2017, Peckner2018, Veech2016}, see also the recent survey articles \cite{survay_sarnak,Ferenczi2018}.
 For example, M\"ullner \cite{Muellner2017} showed that
the Sarnak conjecture holds for all automatic sequences. 

The most prominent automatic sequence is the Thue--Morse sequence $t(n)$, which can be defined in several different ways.
For example, $t(n) = s_2(n) \bmod 2$, where $s_2(n)$ denotes the
number of $1$'s in the binary expansion of $n$.

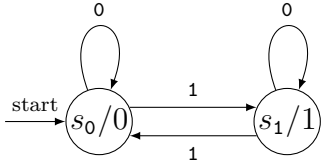
\begin{figure}\label{fig1}
\begin{center}
\begin{tikzpicture}[    description/.style={fill=white,inner sep=2pt},
                        xscale=2.5, yscale=2.5]

  \tikzstyle{vertex}=[circle,draw,fill=black!0,minimum size=15pt,inner sep=0pt]
  \node[vertex] (n0) at (0,0) {$s_{\mathtt 0} / 0$};
  \node[vertex] (n1) at (1,0) {$s_{\mathtt 1} / 1$};

  \path[>=latex,->]
  (-0.5,0) edge node[auto]{\tiny start} (n0)
  (n0) edge [out=110,in=70,looseness=10] node[auto] {$\scriptstyle \mathtt 0$} (n0)
  (n0.25) edge node[auto] {$\scriptstyle \mathtt 1$} (n1.155)
  (n1.205) edge node[auto] {$\scriptstyle \mathtt 1$} (n0.335)
  (n1) edge[out=110,in=70,looseness=10] node[auto] {$\scriptstyle \mathtt 0$} (n1);

\end{tikzpicture}
\end{center}
\caption{An automaton that generates the Thue--Morse sequence}
\end{figure}
The sequence $t(n)$ can be also generated by the automaton displayed in Figure~\ref{fig1}. The input is the binary expansion of $n$. Starting from the initial state $s_0$, the subsequent states
follow the directed edges according to the input sequence. Finally if the input contains an even number of $1$'s then the automaton stops
at the state $s_0$ (and the output is $0$). Conversely if the input contains an odd number of $1$'s then the automaton stops
at the state $s_1$ (and the output is $1$). 

In general, a sequence $a(n)$ is called $k$-\emph{automatic} if it is the output sequence of a finite automaton (where every state has $k$ outgoing edges that are labeled with $0,1,\ldots, k-1$ and where the input is the $k$-ary digital expansion of $n$. Note that in this work we always start with the most significant digit). There are several equivalent definitions, see for example~\cite{Allouche2003}. It should be mentioned that an automatic sequence $a(n)$ can be generated by different automata, however, there is (up to isomorphisms) a unique minimal automaton that generates $a(n)$.

In the present paper we will focus on so-called synchronizing automatic sequences. A $k$-automatic sequence is 
called \emph{synchronizing} if there exists a synchronizing word $w \in \{0,1,\ldots k-1\}^*$, which {\it resets} the
corresponding (minimal) automaton. That is, if we apply $w$ on the automaton then the automaton stops in a specific
state $s$ no matter at which state it was started. For example, a periodic sequence is a synchronizing automatic sequence.
It is worth mentioning that Berlinkov \cite{Berlinkov2016} established
a result showing that almost all automata are synchronizing. Furthermore, it was shown by M\"ullner \cite{Muellner2017} that
every automatic sequence can be represented in terms of a synchronizing automatic sequence and a so-called 
\emph{invertible} automatic sequence. This result was crucial to prove that every automatic sequence satisfies the Sarnak conjecture.
By the way, the Thue--Morse sequence is invertible (and certainly not synchronizing).

The study of (special) automatic sequences along Piatetski-Shapiro sequences $\floor{ n^c }$ has 
actually a long history. Mauduit and Rivat \cite{Mauduit1995, Mauduit2005} showed for the Thue--Morse sequence $t(n)$
(and partly for more general sequences) that the letters $0$ and $1$ appear in $t(\floor{ n^c })$ with
equal asymptotic frequency $1/2$ if $1 < c < 7/5 $. 
This was then generalized to general automatic sequences
by Deshouillers, Drmota, and Morgenbesser \cite{Deshouillers2012} who showed that 
for every $1< c < 7/5$ and for all automatic sequence $a(n)$
the asymptotic logarithmic densities of $a(\floor{ n^c })$ are the same as those of $a(n)$. In the case when one of the sequences admits asymptotic frequencies, then so does the other one and both asymptotic frequencies are equal. 

All results of this paper could be generalized to functions $g(x)$ in a Hardy field \cite{Boshernitzan1994}
of growth $g(x) \sim x^{c+o(1)}$ (for non-integers $c>1$). However, in order to keep the presentation more
readable we restrict ourselves to $g(x) = x^c$.


Our first result generalizes the density result to all non-integers $c>1$ for synchronizing automatic sequences.
\begin{theorem}\label{Th1}
Suppose that $a(n)$ is a synchronizing automatic sequence. Then for every non-integer $c> 1$ the
frequencies of the letters of $a(\floor{ n^c })$ exist and are equal to the corresponding 
frequencies of the letters of $a(n)$.
\end{theorem}

Although we expect a corresponding theorem for all automatic sequences there is certainly a big difference
between synchronizing sequences and, for example, the Thue--Morse sequence. In another recent paper it was shown by
M\"ullner and Spiegelhofer \cite{Muellner2017a} that for $1 < c < 3/2$ the sequence $t(\floor{ n^c })$ is actually
a $2$-normal sequences, that is, not only the letters $0$ and $1$ appear with asymptotic frequency $1/2$ but also all finite $0$-$1$-blocks of length $L\ge 1$ appear with asymptotic frequency $2^{-L}$. 
In particular this shows that the subword complexity of $t(\floor{ n^c })$ is $2^L$ so
that  $t(\floor{ n^c })$ is certainly not deterministic.
Interestingly, such a property is not satisfied for synchronizing automatic sequences.

\begin{theorem}\label{Th3}
Suppose that $a(n)$ is a synchronizing $k$-automatic sequence, where $k\geq 2$. Then for every non-integer $c > 1$ the
subsequence $a(\floor{ n^c})$ is deterministic.
\end{theorem}

For periodic sequences we can be actually more precise.
\begin{theorem}\label{Th2}
Suppose that $a(n)$ is a periodic sequence. Then for every non-integer $c>1$, the subword complexity $N_H$ of $(a\floor{n^c})_n$ is bounded from above by a polynomial in $H$.
\end{theorem}
We will also derive in Proposition~\ref{lbcomplexity}
a very simple (and certainly not optimal) bound of the form $N_H \ge H^2/4$ 
for the sequence $\lfloor n^c \rfloor \bmod m$ (recall that the periodic sequence $a(n) = n \bmod m$ is
a synchronizing automatic sequence).
Since automatic sequences have at most linear subword complexity this observation also shows that 
$\lfloor n^c \rfloor \bmod m$ is not an automatic sequence. 
We also note that our proofs actually can be easily adapted to prove Theorems~\ref{Th2} and \ref{Th3} for all $c > 1$.


Since synchronizing automatic sequences can be approximated by periodic sequences 
it is natural to consider first periodic sequences -- and this will be actually be done in the proof part.
However, it is by no means trivial to transfer the result from Theorem~\ref{Th2} to Theorem~\ref{Th3}.
Interestingly, in order to settle the problem for general synchronizing automatic sequences $a(n)$ we  
consider first $a(n)$ along integer polynomials $P(n)$. More precisely 
if the subword complexity of the union of all subsequences 
$a(P(n))$, where $P(n)$ ranges over all integer polynomials of degree $\le d = \lfloor c \rfloor$,
is subexponential then it follows the subsequence $a(\floor{ n^c})$ is deterministic. In order to show
this property we will use the fact that the $k$-kernel of $k$-automatic sequences is finite
and proper uniform equidistribution properties of ``high'' digits of these polynomials.



\medskip

As mentioned above, it was conjectured by Sarnak that all deterministic sequences are
orthogonal to the M\"obius function, that is, \eqref{eqmoebius} holds. This was already checked for
$a(\floor{ n^c})$ for periodic sequences \cite[Theorem~4]{Deshouillers2019}. 
Our next results provides an even stronger statement.

\begin{theorem}\label{Th4}
Suppose that $a(n)$ is a synchronizing automatic sequence. Then for every non-integer $c> 1$
we have
\[
\sum_{n\le N}  \mu(n) a(\floor{n^c}) = O \left( N \exp\left( -\kappa \frac{ (\log N)^{3/5}}{(\log \log N)^{1/5}} \right) \right)
\] with an absolute constant $\kappa > 0$.

Furthermore there are $\delta > 0$ depending on $a$ and $c$, and $C$ that does not depend on $c$ such that
\[
	\sum_{n\le N}  \Lambda(n) a(\floor{ n^c}) =  C \Psi(N) +  O(N^{1-\delta}),
\]
where $\Lambda(n)$ denotes the von Mangoldt function and
$\Psi$ the Chebyshev function.
\end{theorem}

Since sums of the form $\sum_{n \le N} \Lambda(n) A(n)$ are very close to 
$\log N \sum_{p\le N} A(p)$ (where the sum is taken over all prime numbers $p$), 
Theorem~\ref{Th4} has the following interesting consequence.

\begin{corollary} \label{corr_p_c}
Suppose that $a(n)$ is a synchronizing automatic sequence and $(p_n)$ the sequence of prime numbers.
 Then for every non-integer $c> 1$ the
frequencies of the letters of $a(\floor{ p_n^c})$ exist and are equal to the corresponding 
frequencies of the letters of $a(n)$.
\end{corollary}

\subsection{Plan of the paper} 
In Section~\ref{sec_thm_1} we prove Theorem~\ref{Th1} using the approximation of synchronizing sequences by periodic sequences, and exploring the distribution of $\floor{n^c} \bmod q$ by Erd\H os-Tur\'an argument and exponential sum estimates. A similar idea applies to the subsequences along primes $\floor{p^c} \bmod{q}$ for the proof of Theorem~\ref{Th4} in Section~\ref{sec_thm_4}. In Section~\ref{sec_thm_2} we prove Theorem~\ref{Th2} by estimating the subword complexity of Taylor polynomials of $(n+h)^c$ with a geometric argument. It will be also shown that the contribution to the subword complexity from the error term in the Taylor approximation is negligible. Additionally, in the end of Section~\ref{sec_thm_2} we give a simple proof for the quadratic lower bound of the subword complexity of $\floor{n^c} \bmod{q}$. Finally, Theorem~\ref{Th3} is proven in Section~\ref{sec_thm_3} by a similar reduction of the subword complexity of $a(\floor{n^c})$ to the subword complexity along Taylor polynomials $a(P(n))$ for synchronizing sequences. The proof uses the discrepancy estimates for $P(n) / q$, finiteness of the kernel of synchronizing sequences and its well-approximability by periodic sequences. 

\raggedbottom

\subsection{Notation and standard results}
In this paper, we let 

$\N$, $\Z$, and $\R$ respectively denote the non-negative integers, the rational integers and the real numbers,

$\floor{x}$ (\emph{resp.}\ $\{x\}$) denote the integral part (\emph{resp.} the fractional part) of a real number~$x$,

$\norm{x} = \min_{z \in \Z} |x-z|$ for any real number $x$, 

$\e(x) = \exp(2\pi i x)$ for any real number $x$,

$\Lambda$ denote the von Mangoldt function, defined by $\Lambda(p^k) = \log p$ for prime powers and $\Lambda(n) = 0$ otherwise,

$\Psi$ denote the Chebyshev summatory function of the von Mangoldt function, i.e. $\Psi(N) = \sum_{n \le N} \Lambda(n)$,

$\mu$ denote the M\"obius function, 

$\binom{c}{h} = \frac{c(c-1)\ldots (c-h+1)}{h!}$ for $c$ real and $h$ non-negative integer,

If $m\ge1$ and $n$ are integers, we write $n\bmod m$ to denote the residue class of $n$ modulo $m$ represented 
by the numbers $0,1,\ldots,m-1$. \\

\emph{Landau and Vinogradov notations.} Let $f$ be a complex function and $g$ a real function taking only positive values. The notations $f = O(g)$ or $f \ll g$ are equivalent to the fact that the function $|f|/g$ is bounded. To stress that the bound may depend on a given parameter or set of parameters, say $c$, we write $f = O_c (g)$ or $f \ll_c g$. Similarly, if $f$ is real and positive, we write $f \gg g$ and $f \gg_c g$ when $|g|/f$ is bounded. \\

\emph{Standard results.}  First we mention the asymptotic relation $\Psi(x) \sim x$ as $x$ 
tends to infinity, a property that is equivalent to the \emph{Prime Number Theorem}. 

Next let $\{ x_1, \ldots, x_N \}$ be a finite set of real numbers. Its \emph{discrepancy} is defined by
\[
	D_N (x_1, \ldots, x_N) = \sup_{0 \le \alpha \le \beta \le 1} \biggl| \frac{\# \{ n \le N : \alpha \le \{ x_n \} < \beta \}}{N} - (\beta - \alpha) \biggr|. 
\]
The Erd\H os-Tur\'an inequality states that there exists an absolute constant $C$ such that for any positive integer $K$, one has
\[
	D_N (x_1, \ldots, x_N) \le C \biggl(  \frac{1}{K} + \frac{1}{N} \sum_{k=1}^K \frac{1}{k} \biggl| \sum_{n=1}^N e(k x_n) \biggr| \biggr).
\] 
A proof can be found in~\cite[Chapter~2]{Kuipers1974}.

Finally we mention the simple property that if $x$ is a real number and $0 \le u < m$ are integers then we have
\begin{align} \label{frac_part_condition}
	\floor{x} \equiv u \bmod{m} \quad \Longleftrightarrow \quad \frac{u}{m} \le \left\{ \frac{x}{m} \right\} < \frac{u+1}{m}.
\end{align}

\section{Proof of Theorem~\ref{Th1}}
\label{sec_thm_1}

The proof of Theorem~\ref{Th1} is actually the most direct one and relies mainly on the fact that synchronizing $k$-automatic sequences are {\it almost periodic} in the following sense.

By \cite[Lemma 2.2]{Deshouillers2015} there are at most $k^{n(1-\eta)}$ words of length $n$ that are not synchronizing (where $\eta > 0$). In particular this means that for every $n_1\ge 0$ there exists a set $U \subseteq \{0,1,\ldots, k-1\}^{n_1}$ of size $|U| \ge k^n - k^{n_1(1-\eta)}$ such that for all $u\in U$ and
$n \equiv u \bmod k^{n_1}$ we have $a(n) = a(u)$. 

We first use this property in order to show that the letters $\alpha \in \mathcal{A}$ have asymptotic frequencies
\begin{align}\label{eqNlimit}
\lim_{N\to\infty} \frac 1N \# \{ n< N : a(n) = \alpha\} 
\end{align}
	By using the above-mentioned property we get 
\begin{align*}
	| \{ n< N : a(n) = \alpha\} |  &= \sum_{u\in U,\, a(u) = \alpha }  | \{ n< N : n \equiv u \bmod k^{n_1}\} | \\ 
	&+ \sum_{u\not\in U}  | \{ n< N : a(n) = \alpha, n \equiv u \bmod k^{n_1} \}  | \\
	&= \sum_{u\in U,\, a(u) = \alpha }  \left( \frac N{k^{n_1}} + O(1) \right) + O\left( \frac{k^{n_1(1-\eta)} N}{k^{n_1}} \right) \\
	&= \frac N{k^{n_1}} | \{ u< k^{n_1} : a(u) = \alpha\} | + O(k^{n_1}) + O(N k^{-\eta n_1} ).
\end{align*}
Hence, the sequence of mean values
\[
	\frac 1{k^{n}} \bigl| \{ u< k^{n} : a(u) = \alpha\} \bigr|
\] is a Cauchy sequence that has a limit that we denote by $\vartheta_\alpha$. The same calculation shows that
the limit~\eqref{eqNlimit} exists and equals $\vartheta_\alpha$. Finally we also get an 
upper bound for the speed of convergence:
\begin{align}\label{eqkn1limit}
	\frac 1{k^{n_1}} | \{ u< k^{n_1} : a(u) = \alpha\} | =  \vartheta_\alpha + O(k^{-\eta n_1} ) .
\end{align}
This also implies that 
\[
	\frac 1{k^{n_1}} \sum_{u < k^{n_1}}  a(u)  =  \sum_{\alpha\in \mathcal{A}} \alpha\vartheta_\alpha + O(k^{-\eta n_1} ).
\]

Next we generalize this calculation for the subsequence $a(\floor{ n^c})$. We need to show the equidistribution of $\floor{ n^c}$ in arithmetic progressions $u \bmod m$ for $m \le N^{\theta}$ for some small fixed $\theta > 0$. The standard approach to this problem requires a non-trivial upper bound for the exponential sum of the form
\[
	\sum_{n \le N} e\left( A n^c \right) \ll N^{1-\delta}
\] with some $\delta = \delta(c) > 0$. Such estimate can be obtained by van der Corput $k$-th derivative test (see, for example,~\cite[Theorem~8.4]{Iwaniec-Kowalski}) with the appropriate choice of~$k$.

\begin{lemma} \label{Le2.1}
	Let $c > 1$ be a non integral real number. There exist two positive constants $\alpha$ and $\eta$ such that, uniformly for $N \ge 1$ and $A$ satisfying $N^{-\eta} \le A \le N^{3c}$ we have
	\[
		\sum_{n \le N} e(A n^c) \ll N^{1-\alpha}.
	\]
\end{lemma} Using~\eqref{frac_part_condition}, the Erd\H os-Tur\'an inequality and Lemma~\ref{Le2.1}, we get the equidistribution property for $(\floor{n^c})$ in arithmetic progressions which we state in the next proposition. This is a consequence of Th\'eor\`eme~2 announced in~\cite{Deshouillers1973}. For $c > 3/2$ the result also follows from~\cite[Proposition~19]{BBBSW}. We have
\begin{proposition} \label{Le2.2}
	Let $c > 1$ be a non integral real number. There exist two positive constants $\theta$ and $\delta$ such that
	\[
		\bigl|\{ n \le N : \floor{ n^c} \equiv u \bmod{m} \} \bigr| = \frac Nm + O\left( \left( \frac Nm \right)^{1-\delta} \right)
	\] uniformly for $1\le m \le N^\theta$ and $0 \le u < m$.
\end{proposition}

\begin{proof}
	We consider the sequence $(n^c / m)_{n \le N}$. By~\eqref{frac_part_condition} and the definition of the discrepancy, we have
	\[
		\left| \bigl| \{ n \le N : \floor{n^c} \equiv u \bmod{m} \} \bigr| - \frac Nm \right| \le N D_N \left( \frac{1^c}{m}, \ldots, \frac{N^c}{m}\right). 	
	\] By the Erd\H os-Tur\'an inequality, we have for any $K \ge 1$
	\[
		ND_N \left( \frac{1^c}{m}, \ldots, \frac{N^c}{m} \right) \ll \left( \frac{N}{K} + \sum_{k=1}^K \frac{1}{k} \biggl| \sum_{n=1}^N e(kn^c/m) \biggr| \right).
	\] With the notation of Lemma~\ref{Le2.1}, we choose $K = \floor{N^{\alpha}}$. For $1 \le k \le K$ and $1 \le m \le N^{\eta}$, we have $N^{-\eta} \le k/m \le N^{3c}$ and we can apply Lemma~\ref{Le2.1}. This leads to
	\[
		\frac{N}{K} + \sum_{k=1}^K \frac{1}{k} \biggl| \sum_{n=1}^N e\left( \frac{kn^c}{m} \right) \biggr| \ll N^{1-\alpha} + N^{1-\alpha} \log N \ll N^{1-\alpha} \log N,
	\] uniformly for $1 \le m \le N^{\eta}$. We now select $\delta = \alpha/2$ and $\theta = \min(\eta, \alpha /2)$ to end the proof. 
\end{proof}

With the help of Proposition~\ref{Le2.2} we immediately obtain (with $k^{n_1} \approx N^\theta$)
\begin{align*}
| \{ n < N : a(\floor{ n^c}) = \alpha \}|  
&= \sum_{u \in M, \, a(u) = \alpha} |\{ n < N : \floor{ n^c} \equiv u \bmod k^{n_1} \}| \\
&+ \sum_{u\not \in M} |\{ n < N : a(\floor{ n^c}) = \alpha,\, \floor{ n^c} \equiv u \bmod k^{n_1} \}| \\
&= \sum_{u \in M, \, a(u) = \alpha} \left( \frac N{k^{n_1}} + O\left( \left( \frac N{k^{n_1}} \right)^\delta \right) \right) \\
&+ O\left( k^{n_1(1- \eta)} \frac N{k^{n_1}} \right) \\
&= \frac N{k^{n_1}} \left( | \{ u < k^{n_1} : a(u) = \alpha \}|  + O\left( k^{n_1(1-\eta )} \right) \right) \\
&+ O\left( \frac N{k^{\eta n_1}} \right) + O\left( k^{n_1} \left( \frac N{k^{n_1}} \right)^\delta \right).
\end{align*}
Thus, by using~\eqref{eqkn1limit} we get
\begin{align*}
\frac 1N | \{ n < N : a(\floor{ n^c}) = \alpha \}| 
&= \vartheta_\alpha + O\left( k^{-\eta n_1} \right) + O\left( \left( \frac{k^{n_1}}N \right)^{1-\delta} \right) \\
&= \vartheta_\alpha + O\left( N^{-\theta \eta} \right) + O\left( N^{-(1-\theta)(1-\delta)} \right),
\end{align*}
which completes the proof of Theorem~\ref{Th1}.

\section{Proof of Theorem~\ref{Th4}}
\label{sec_thm_4}

Similarly to the proof of Theorem~\ref{Th1} we first show the equidistribution of $\floor{ p^c}$ in arithmetic progressions $u\bmod{m}$. This requires an upper bound for the exponential sum over primes:
\[
	\sum_{n \le N} \Lambda(n) e\left( A n^c \right) \ll N^{1-\delta},
\] where $\delta = \delta(c) > 0$. Such estimates have repeatedly appeared in the literature. See, for example,~\cite{Baker-Kolesnik, Cao-Zhai, Changa, Shubin}. 

\begin{lemma} \label{Le3.1}
	Let $c > 1$ be a non integral real number. There exist two positive real numbers $\delta$ and $\theta$ such that uniformly for any $A$ in $[N^{-\theta}, N^{2\theta}]$, we have 
	\begin{align} \label{expsum_primes}
		\sum_{n \le N} \Lambda(n) e\left( A n^c \right) \ll N^{1-\delta}.
	\end{align}
\end{lemma} For the proof of this result see Lemma~2 in~\cite{Changa}. Note that here we do not aim at the best value of exponent $\delta$.

\begin{remark}
	In Lemma~\ref{Le3.1} one clearly cannot choose $\theta$ as large as $\norm{c}$ as $c$ varies.
\end{remark}

Using the bound~\eqref{expsum_primes} we obtain an asymptotic formula for the number of primes with $\floor{p^c} \equiv u\bmod{m}$.

\begin{proposition} \label{PNT}
	Let $c > 1$ be a non integral real number. There exist two positive constants $\theta$ and $\delta$ such that
	\[
		\sum_{\substack{n \le N \\ \floor{ n^c } \equiv u\bmod{m}}} \Lambda(n) = \frac{\Psi(N)}{m} + O\left( \left( \frac{N}{m} \right)^{1-\delta} \right)
	\]
	uniformly for $1 \le m \le N^{\theta}$ and $0 \le u < m$.
\end{proposition} The proof is similar to the proof of Proposition~\ref{Le2.2}.

\begin{remark}
	The same approach gives the formula
	\begin{align} \label{eqmurelation}
		\sum_{\substack{n \le N \\ \floor{ n^c } \equiv u\bmod{m}}} \mu(n) = \frac{M(N)}{m} + O\left( \left( \frac{N}{m} \right)^{1-\delta} \right)
	\end{align}
	with
	\[
	M(N) = \sum_{n \le N} \mu(n).
	\] This would follow from a bound similar to~\eqref{expsum_primes}, where $\Lambda(n)$ is replaced by $\mu(n)$.  
\end{remark}

We are now ready to complete the proof of Theorem~\ref{Th4}.
We start with the decomposition used above, where $k^{n_1} = N^\theta$ and $0 < \theta < 1$, and proceed as follows.
\begin{align*}
	\sum_{n<N} \Lambda(n) a(\floor{n^c}) 
	&= \sum_{u\in U} \sum_{n< N, \,  \floor{ n^c } \equiv u \bmod k^{n_1}} \Lambda(n)  a(\floor{ n^c })  \\
	&+ \sum_{u\not \in U} \sum_{n< N, \,  \floor{ n^c } \equiv u \bmod k^{n_1}} \Lambda(n)  a(\floor{ n^c }) \\
	&= \sum_{u\in U} a(u) \sum_{n< N, \,  \floor{ n^c } \equiv u \bmod k^{n_1}} \Lambda(n)  \\ \quad 
	&+ O\left( k^{n_1(1-\eta)} \max_{u < k^{n_1}}  \sum_{n< N, \,  \floor{ n^c } \equiv u \bmod k^{n_1}} \Lambda(n) \right) \\
	&= \sum_{u\in U} a(u) \left(  \frac{\Psi(N)}{k^{n_1}}  + O\left(   \left( \frac{N}{k^{n_1}} \right)^{1-\delta} \right) \right)
	+ O\left(  \frac{k^{n_1(1-\eta)} N}{k^{n_1}} \right) \\
	&= \frac 1{k^{n_1}} \sum_{u < k^{n_1}} a(u) \, \Psi(N) + O(N k^{-\eta n_1}) + O(N^{1-\delta} k^{n_1\delta})\\
	&= \Psi(N) \sum_{\alpha\in \mathcal{A}} \alpha\vartheta_\alpha  + O(N k^{-\eta n_1}) + O(N^{1-\delta} k^{n_1\delta}) \\
	&=  \Psi(N) \sum_{\alpha\in \mathcal{A}} \alpha\vartheta_\alpha + O(N^{1-\eta\theta})  + O(N^{1-\delta(1-\theta)} ).
\end{align*}
This proves the second part of Theorem~\ref{Th4} for $\Lambda(n)$ and 
\[
	C = \sum_{\alpha\in \mathcal{A}} \alpha\vartheta_\alpha .
\]

The orthogonality of $\mu(n)$ and $a(\floor{n^c})$ is proved in the same way. We just apply~\eqref{eqmurelation} instead of Proposition~\ref{PNT} and the Vinogradov--Korobov \cite{Vinogradov_PNT, Korobov} bound for $M(N)$.

Corollary~\ref{corr_p_c} follows from a similar computation for
\[
	|\{ p < N: a(\floor{p^c}) = \alpha \}|
\] and summation by parts.

\begin{remark}
	The result can of course be generalized to the sequences of the form $a(\floor{f(n)})$ for any arbitrary smooth function $f(x)$ satisfying van der Corput restrictions on the size of the derivatives of $f$. 
\end{remark}

\section{Proof of Theorem~\ref{Th2}}
\label{sec_thm_2}

The proof of Theorem~\ref{Th2} is divided into several steps.
We start with the subword complexity of $\floor{ n^c} \bmod m$.

\subsection{Decomposition}

Let $c>1$ be a non-integer and $m \geq 1$ be a fixed integer. We denote by $d = \floor{c}$.
We will again make use of the property~\eqref{frac_part_condition} and also of the relation 
\begin{align} \label{frac_part_property_2}
		\cb{a+b} = \cb{\cb{a} + \cb{b}}.
\end{align} For any $n\geq 0$, we write
\begin{align*}
	A_t^{(n)} = \binom{c}{t} n^{c-t}  \qquad \text{for } 0\leq t \leq d
\end{align*}
and
\begin{align*}
	P^{(n)}(h) = \sum_{t= 0}^{d} A_t^{(n)} h^t.
\end{align*}
By Taylor expansion we find
\begin{align*}
	(n+h)^c = P^{(n)}(h) + f^{(n)}_h,
\end{align*}
where $f^{(n)}_h$ denotes the error in the Taylor expansion. One easily sees that the Taylor remainder
\begin{align} \label{eq_error_1}
	f_h^{(n)} = \binom{c}{d+1} (n+\theta h)^{c-d-1} h^{d+1},
\end{align} valid for some $\theta$ in $(0, 1)$ implies that $f_h^{(n)}$ is always non-negative and small when $n$ is large enough in terms of $H$. 
Putting everything together we find that $\floor{(n+h)^c} \equiv u \bmod m$ (with $0 \leq u < m$ and $0\le h < H$)
if and only if 
\begin{align*}
	\cb{\frac{ P^{(n)}(h) + f^{(n)}_h}{m}} &\in \left[\frac{u}{m}, \frac{u+1}{m}\right).
\end{align*}

\subsection{Subword complexity of the Taylor approximation of $n^c$ without error term}

With the help of this notation we obtain a subword complexity bound for the leading term $\floor{P^{(n)}(h)}$.
\begin{proposition}
	We have uniformly in $m$ and $H$,
	\[
		\abs{\left\{ \rb{ \floor{P^{(n)}(0)} \bmod m , \ldots, \floor{P^{(n)}(H-1)} \bmod m   }   : n\ge 0 \right\}}
			 \ll_d m^{d+1} H^{(d+1)(d+2)}.
	\]
\end{proposition}
\begin{proof}
	By~\eqref{frac_part_condition} and~\eqref{frac_part_property_2} we have that $P^{(n)}(h) \equiv u_h \bmod m$ if and only if
	\begin{align}\label{eq_no_error_term}
		\cb{\sum_{t=0}^{d} \cb{\frac{A_t^{(n)}}{m}} h^t} \in \left[\frac{u_h}{m}, \frac{u_h+1}{m}\right).
	\end{align}
	We consider the $mH^{d+1}$ intervals
	\begin{align*}
		\left[0, \frac{1}{m}\right), \ldots, \left[\frac{mH^{d+1}-1}{m}, H^{d+1}\right),
	\end{align*}
	and let us call them $I_1, \ldots, I_{mH^{d+1}}$.
	For each $h \in [0, H-1]$ and each $i \in [1, mH^{d+1}]$, the set of $(d+1)$-tuples $(x_0, x_1, \ldots, x_d)$ in $\R^{d+1}$ such that
	\begin{align*}
		x_0 + x_1 h + x_2 h^2 + \ldots + x_d h^d \in I_i
	\end{align*}
	is a strip in $\R^{d+1}$ defined by two parallel hyperplanes, namely
	\begin{align*}
		\frac{i-1}{m} \leq x_0 + x_1 h + \ldots + x_d h^d < \frac{i}{m}.
	\end{align*}
	If we consider simultaneously the different values of $h$ in $[0, H-1]$, we are separating $\R^{d+1}$ into regions the sides of which belong to a family of $H \cdot m H^{d+1}$ hyperplanes (see Figure~\ref{fig_2}). 
	\begin{figure}
\begin{tikzpicture}
	\tkzInit[xmax=1,ymax=1,xmin=0,ymin=0, xstep = 0.25, ystep = 0.25]
	\tkzDrawX[label = $x_0$]
	\tkzDrawY[label = $x_1$]
   	\draw (0,0) -- (0,4);
      	\draw (2,0) -- (2,4);
	\draw (4,0) -- (4,4);

	\draw (2,4) -- (4,2);
	\draw (0,4) -- (4,0);
	\draw (0,2) -- (2,0);
	
	\draw (0,1) -- (2,0);
	\draw (0,2) -- (4,0);
	\draw (0,3) -- (4,1);
	\draw (0,4) -- (4,2);
	\draw (2,4) -- (4,3);
	
	\draw[line width=0 pt, fill=gray] (2,1) -- (4,0) -- (2,2) --cycle;
\end{tikzpicture}
\caption{The gray area corresponds to the intersection of the strips $0.5 < x_0 < 1$, $0.5 < x_0 + x_1 < 1$ and $1 < x_0 + 2 x_1 < 1.5$.}
\label{fig_2}
\end{figure}
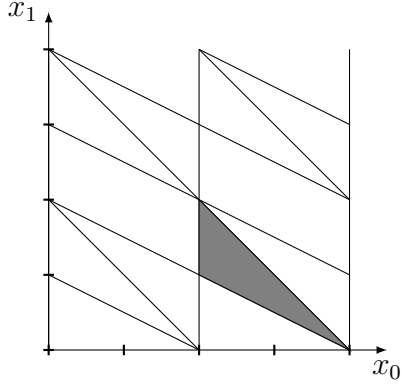
	
	The following observations finish the proof of the proposition.
	\begin{enumerate}
		\item Let us fix one such region. Then the sequence 
			\begin{align*}
				\rb{ \floor{P^{(n)}(0)} \bmod m , \ldots, \floor{P^{(n)}(H-1)} \bmod m   } 
			\end{align*}
			is constant when $\rb{\cb{A_0^{(n)}/m}, \ldots, \cb{A_d^{(n)}/m}}$ belongs to this region.
		\item The number of regions (in this context often called cells) is polynomial in $H$. Indeed, it is classical (see for example~\cite{Stanley2007}), that it can be bounded from above by
		\begin{align*}
			\sum_{i = 0}^{d+1} { m H^{d+2}\choose i} &\leq \sum_{i=0}^{d+1} (m H^{d+2})^{i} \ll_d m^{d+1} H^{(d+1)(d+2)}.
		\end{align*}
		
	\end{enumerate}
\end{proof}



\subsection{Setup for the subword complexity of $\floor{n^c} \bmod m$}

It is, however, not immediate to transfer the subword complexity bound for $\floor{P^{(n)}(h)}$ to
$\floor{(n+h)^c}$. 
We use a very similar decomposition to equation~\eqref{eq_no_error_term}, namely $\floor{(n+h)^c} \equiv u_h \bmod m$ if and only if
\begin{align}\label{eq_detection_error}
	\cb{Q^{(n,m)}(h) + \frac{f_h^{(n)}}{m}} \in \left[ \frac{u_h}{m}, \frac{u_h + 1}{m} \right),
\end{align}
where $Q^{(n,m)}(h) = \sum_{t=0}^{d} \cb{A_t^{(n)}/m} h^t$.

	We only consider $n$ to be large enough, such that $0\leq f_h^{(n)} <1$ (this is satisfied when $n \gg_c H^{(d+1)/(d+1-c)}$), and we can rewrite equation~\eqref{eq_detection_error} as
    \begin{align*}
		Q^{(n,m)}(h) + \frac{f_h^{(n)}}{m}& \in \left[\frac{u_h}{m}, \frac{u_h+1}{m}\right) + z_h 
    \quad \text{for some } z_h \in \{0, 1, \ldots, H^{d+1}\}.
	\end{align*}
	This is obviously equivalent to $\rb{\cb{A_0^{(n)}/m}, \ldots, \cb{A_d^{(n)}/m}}$ belonging to the strip defined by
	\begin{align}\label{eq_strip}
		\frac{u_h}{m} + z_h - \frac{f^{(n)}_h}{m} \leq x_0 + x_1 h + \ldots + x_d h^d < \frac{u_h+1}{m} + z_h - \frac{f^{(n)}_h}{m}.
	\end{align}

	We denote the strip defined by~\eqref{eq_strip} as $S(h; u_h, z_h, f_h^{(n)})$.
	Moreover, we define the limiting upper and lower hyperplanes of the strip as the solutions in $(x_0, \ldots, x_{d})$ of the following equations respectively:
	\begin{align*}
		 x_0 + x_1 h + \ldots + x_d h^d &= \frac{u_h+1}{m} + z_h - \frac{f^{(n)}_h}{m}\\
		 x_0 + x_1 h + \ldots + x_d h^d &=\frac{u_h}{m} + z_h - \frac{f^{(n)}_h}{m} .
	\end{align*}
	The above discussion can be captured by the following lemma.
	\begin{lemma}
		Let $H, m \in \N$ and $u_0,\ldots, u_{H-1} \in \{0, \ldots, m-1\}$.
		Then for any $n\in \N$ large enough in terms of $H$ (again it is sufficient to require $n \gg_c H^{(d+1)/(d+1-c)}$),
		\begin{align*}
			\rb{{\floor{ n^c} \bmod m}, {\floor{(n+1)^c}\bmod m}, \ldots, {\floor{(n+H-1)^c}\bmod m}} = \rb{u_0, \ldots, u_{H-1}}
		\end{align*}
		if and only if there exist $z_0, \ldots, z_{H-1} \in \{0, \ldots, H^{d+1}\}$ such that $(A_0^{(n)}, \ldots, A_d^{(n)})$ belongs to the intersection of the strips $S(h; u_h, z_h, f_h^{(n)})$ for $h = 0, \ldots, H-1$.
	\end{lemma}
	We denote the intersection of the strips $S(h; u_h, z_h, f_h^{(n)})$ for $h = 0, \ldots, H-1$ by $IS(\mathbf{u}, \mathbf{z}, \mathbf{f^{(n)}})$, where $\mathbf{u} = (u_0, \ldots, u_{H-1})$, etc.

	Next we use~\eqref{eq_error_1} to give the possible values of $f_h^{(n)}$ some structure.
	\begin{lemma}\label{le_belongs_to_E}
		Let $\kb, \kc \in \N$ and $\ka \geq \max (\kb/(d+1-c), \kc+1)$. Then there exists an implied constant only depending on $c$ such that for $n\gg_c H^{\ka}$ we have that $(f_0^{(n)}, \ldots, f_{H-1}^{(n)})$ belongs to the set
		\begin{align*}
			\mathcal{E} = \mathcal{E}(\kb, \kc) \eqdef \{(\varepsilon \cdot h^{d+1}\cdot (1 + g_h))_{h\in \{0, \ldots, H-1\}}: \varepsilon \in (0, H^{-\kb}), \abs{g_h} \leq H^{-\kc}\}.
		\end{align*}
	\end{lemma}
	\begin{proof}
	Comparing this to~\eqref{eq_error_1}, we see that we can choose $\varepsilon = \frac{c \ldots (c-d)}{(d+1)!} \cdot \frac{1}{n^{d+1-c}}$. Here the first factor only depends on $c$ so that it suffices to choose $n \gg_c H^{\kb/(d+1-c)}$. For the error term $g_h$ to be small enough, we need $n \gg_c H^{\kc+1}$.
	\end{proof}
	
	\begin{lemma}
		If 
		$IS(\mathbf{u}, \mathbf{z}, \mathbf{f})$ is non-empty, then it contains an open ball, i.e. it has non-empty interior and positive volume.
	\end{lemma}
	\begin{proof}
		Let us consider a point $(y_0, \ldots, y_{d})$ that belongs to $IS(\mathbf{u}, \mathbf{z}, \mathbf{f})$. We easily see that for each $h \in \{0, \ldots, H-1\}$ there exists some $\varepsilon_h$ such that every $(y_0 + w_0, \ldots, y_d + w_d)$ belongs to the strip $S(h; u_h, z_h, f_h)$ for $0\leq w_0, \ldots, w_d < \varepsilon_h$.
		The proof follows immediately.
	\end{proof}
	
	\begin{corollary}
		Let $n$ be such that $(f_0^{(n)}, \ldots, f_{H-1}^{(n)}) \in \mathcal{E}$. If 
		\begin{align*}
			\rb{{\floor{ n^c}\bmod m}, {\floor{(n+1)^c}\bmod m}, \ldots, {\floor{(n+H-1)^c}\bmod m}} = \rb{u_0, \ldots, u_{H-1}}
		\end{align*}
		then there exist $z_0, \ldots, z_{H-1} \in \{0, \ldots, H^{d+1}\}$ and $(f_0, \ldots, f_{H-1}) \in \mathcal{E}$ such that $IS(\mathbf{u}, \mathbf{z}, \mathbf{f})$ forms a polyhedron in $\R^{H}$ with non-empty interior, i.e. it has a positive volume.
	\end{corollary}
	

\subsection{Proof of Theorem~\ref{Th2}}
The goal is to show that for two different choices of elements $\mathbf{f}, \mathbf{f}' \in \mathcal{E}$ the polyhedron $IS(\mathbf{u}, \mathbf{z}, \mathbf{f})$ is non-empty if and only if $IS(\mathbf{u}, \mathbf{z}, \mathbf{f}')$ is non-empty (see Figure~\ref{fig_3}). This shows that we only need to consider one choice of $(f_0, \ldots, f_{H-1}) \in \mathcal{E}$ which can be treated analogously to the case without an error term.
	
\begin{figure}
\scalebox{0.8}{
	\begin{tikzpicture}
	\tkzInit[xmax=1,ymax=1,xmin=0,ymin=0, xstep = 0.25, ystep = 0.25]
	\tkzDrawX[label = $x_0$]
	\tkzDrawY[label = $x_1$]
   	\draw (0,0) -- (0,4);
      	\draw (2,0) -- (2,4);
	\draw (4,0) -- (4,4);

	\draw (4-0.223, 4) -- (4, 4-0.223);
	\draw (2-0.223,4) -- (4,2-0.223);
	\draw (0,4-0.223) -- (4-0.232,0);
	\draw (0,2-0.223) -- (2-0.223,0);
	
	\draw (0,1-0.459) -- (2-0.918,0);
	\draw (0,2-0.459) -- (4-0.918,0);
	\draw (0,3-0.459) -- (4,1-0.459);
	\draw (0,4-0.459) -- (4,2-0.459);
	\draw (2-0.918,4) -- (4,3-0.459);
	\draw (4-0.918, 4) -- (4, 4-0.459);
\end{tikzpicture}
\quad
\begin{tikzpicture}
	\tkzInit[xmax=1,ymax=1,xmin=0,ymin=0, xstep = 0.25, ystep = 0.25]
	\tkzDrawX[label = $x_0$]
	\tkzDrawY[label = $x_1$]
   	\draw (0,0) -- (0,4);
      	\draw (2,0) -- (2,4);
	\draw (4,0) -- (4,4);

	\draw (4-0.166, 4) -- (4, 4-0.166);
	\draw (2-0.166,4) -- (4,2-0.166);
	\draw (0,4-0.166) -- (4-0.166,0);
	\draw (0,2-0.166) -- (2-0.166,0);
	
	\draw (0,1-0.330) -- (2-0.660,0);
	\draw (0,2-0.330) -- (4-0.660,0);
	\draw (0,3-0.330) -- (4,1-0.330);
	\draw (0,4-0.330) -- (4,2-0.330);
	\draw (2-0.660,4) -- (4,3-0.330);
	\draw (4-0.660, 4) -- (4, 4-0.330);
\end{tikzpicture}
\quad
\begin{tikzpicture}
	\tkzInit[xmax=1,ymax=1,xmin=0,ymin=0, xstep = 0.25, ystep = 0.25]
	\tkzDrawX[label = $x_0$]
	\tkzDrawY[label = $x_1$]
   	\draw (0,0) -- (0,4);
      	\draw (2,0) -- (2,4);
	\draw (4,0) -- (4,4);

	\draw (4-0.074, 4) -- (4, 4-0.074);
	\draw (2-0.074,4) -- (4,2-0.074);
	\draw (0,4-0.074) -- (4-0.074,0);
	\draw (0,2-0.074) -- (2-0.074,0);
	
	\draw (0,1-0.149) -- (2-0.298,0);
	\draw (0,2-0.149) -- (4-0.298,0);
	\draw (0,3-0.149) -- (4,1-0.149);
	\draw (0,4-0.149) -- (4,2-0.149);
	\draw (2-0.298,4) -- (4,3-0.149);
	\draw (4-0.298, 4) -- (4, 4-0.149);
\end{tikzpicture}
}
\caption{The shifted hyperplanes for $m = 2, c = 1.5, H = 2$ and $n = 10, 20$ and $n = 100$ respectively.}
\label{fig_3}
\end{figure}
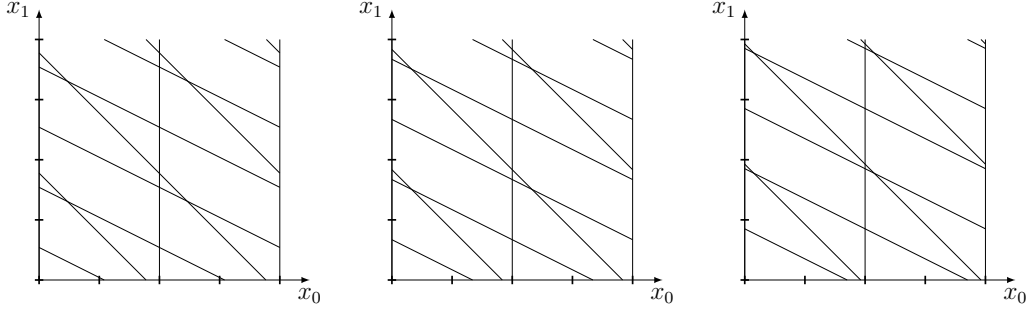
	
	It remains to show that $IS(\mathbf{u}, \mathbf{z}, \mathbf{f})$ is non-empty if and only if $IS(\mathbf{u}, \mathbf{z}, \mathbf{f}')$ is non-empty. If we assume on the contrary that $IS(\mathbf{u}, \mathbf{z}, \mathbf{f})$ is empty, then the strategy is to find a third choice $(f_0'', \ldots, f_{H-1}'')\in \mathcal{E}$ such that for at least $d+2$ of the limiting hyperplanes of the strips $S(h; u_h, z_h, f_h'')$ meet in a single point (Proposition~\ref{pr_intersection}). Then we show that this is impossible in Proposition~\ref{pr_intersection_impossible}. We will start by proving Proposition~\ref{pr_intersection}.

	\begin{proposition}\label{pr_intersection}
		Suppose that there exist $(u_0, \ldots, u_{H-1})$ and $(z_0, \ldots, z_{H-1})$ and two different error terms $(f_0, \ldots, f_{H-1}), (f_0', \ldots, f_{H-1}') \in \mathcal{E}$ such that the intersection of the strips $IS(\mathbf{u}, \mathbf{z}, \mathbf{f})$ has empty interior and the intersection of the strips $IS(\mathbf{u}, \mathbf{z}, \mathbf{f}')$ has non-empty interior.
		Then, there exists some $(f_0'', \ldots, f_{H-1}'')\in \mathcal{E}$ such that $d+2$ of the limiting hyperplanes of the strips $S(h; u_h, z_h, f_h'')$ intersect in a single point.
	\end{proposition}
	\begin{proof}
		Since $\mathcal{E}$ is a connected region, as a Cartesian product of intervals, there exists a continuous path $p: [0,1] \to \mathcal{E}$ such that $p(0) = (f_0, \ldots, f_{H-1})$ and $p(1) = (f_0', \ldots, f_{H-1}')$. To any $\xi \in [0,1]$ we define $v(\xi)$ as the volume of $IS(\mathbf{u}, \mathbf{z}, \mathbf{p}(\xi))$. We see directly that $v$ is a continuous function from $[0,1]$ to $\R_{\geq 0}$ for which $v(0) =0$ and $v(1) > 0$. 
		Thus, we can find a sequence $\xi_n \in [0,1]$ that converges to some $\xi \in [0,1]$ such that $v(\xi) = 0$ and $v(\xi_n) > 0$ for $n\in \N$.
		We recall that $IS(\mathbf{u}, \mathbf{z}, \mathbf{p}(\xi_n))$ is a polyhedron whose corners (and facets) are the intersection of $d$ (or less) hyperplanes of the form
		\begin{align*}
			x_0 + x_1 h_i + \ldots + x_d h_i^d = \frac{u_{h_i} + x_i}{m} + z_{h_i} - \frac{f_{h_i}}{m},
		\end{align*}
		where $x_i \in \{0,1\}$ depending on whether the hyperplane with index $i$ is an ``upper'' or ``lower'' limiting hyperplane.
		As there are only finitely many choices for the $h_i$ and $x_i$ we can assume without loss of generality that they are the same for all $\xi_n$.
		Since the orientations of the hyperplanes are independent of $\mathbf{f}$, the only way that the volume can converge to zero is if the distance of at least two different corners converges to zero. Thus, the limit (for $n\to \infty$) of any such two corners is a point that is the intersection of at least $d+1$ limiting hyperplanes, which finishes the proof.
	\end{proof}

	\begin{proposition}\label{pr_intersection_impossible}
		Let $\kb = \frac{d^2 + 7d+4}{2}, \kc = d+2$ and let $H$ be sufficiently large in terms of $c$.
		Then for no $(f_0, \ldots, f_{H-1}) \in \mathcal{E}(\kb, \kc)$ there exist $u_h$ and $z_h$ such that $d+2$ of the limiting hyperplanes of the strips $S(h; u_h, z_h, f_h)$ intersect in a single point.
	\end{proposition}
	\begin{proof}
		Let us assume that there exists a point $\mathbf{y} = (y_0, \ldots, y_d)^{T}$ which is the intersection of $d+2$ such limiting hyperplanes. We can assume without loss of generality that all of these limiting hyperplanes are lower limiting hyperplanes. Thus, $\mathbf{y}$ satisfies the equations
		\begin{align*}
			x_0 + x_1 h_j + \ldots + x_d h_j^d = u_{h_j} + \frac{z_{h_j}}{m} - \frac{f_{h_j}}{m}
		\end{align*}
		for different $h_1,\ldots, h_{d+2} \in \{0, \ldots, H-1\}$.
		
		We consider the ``unshifted'' version of the hyperplanes, i.e. the hyperplanes defined for $h_1, \ldots, h_{d+2} \in \{0, \ldots, H-1\}$ via
		\begin{align}\label{eq_unshifted}
			x_0 + x_1 h_j + \ldots + x_d h_j^d = u_{h_j} + \frac{z_{h_j}}{m}.
		\end{align}
		We distinguish between the following two cases.
		
		\textbf{Case 1}: The ``unshifted'' hyperplanes~\eqref{eq_unshifted} intersect in a single point.\\
		We first introduce some new notation. We denote by $\mathbf{y}'$ the intersection of the ``unshifted'' hyperplanes.
		Moreover, we denote by $s_j = u_{h_j} + \frac{z_{h_j}}{m}$ and 
		\[
			\mathbf{s}_{(1)} = (s_1, \ldots, s_d, s_{d+1})^{T}, \qquad  \mathbf{s}_{(2)} = (s_1, \ldots, s_d, s_{d+2})^{T}
		\] two vectors that differ only in the last coordinate.
		Thus, we find that $\mathbf{y}'$ satisfies the following systems of linear equations,
		\begin{align*}
			V(h_1, \ldots,h_d, h_{d+1}) \cdot \mathbf{y}' &= \mathbf{s}_{(1)}\\
			V(h_1, \ldots, h_d, h_{d+2}) \cdot \mathbf{y}' &= \mathbf{s}_{(2)},
		\end{align*}
		where $V(h_1, \ldots, h_{d+1})$ denotes the Vandermonde matrix, of which the $i,j$-th entry is given by $h_i^{j-1}$.
		If we let $\mathbf{f}_{(1)} = \rb{\frac{f_{h_1}}{m}, \ldots, \frac{f_{h_d}}{m}, \frac{f_{h_{d+1}}}{m}}^{T}$ and $\mathbf{f}_{(2)} = \rb{\frac{f_{h_1}}{m}, \ldots, \frac{f_{h_d}}{m}, \frac{f_{h_{d+2}}}{m}}^{T}$, we find analogously
		\begin{align*}
			V(h_1, \ldots,h_d, h_{d+1}) \cdot \mathbf{y} &= \mathbf{s}_{(1)} - \mathbf{f}_{(1)}\\
			V(h_1, \ldots, h_d, h_{d+2}) \cdot \mathbf{y} &= \mathbf{s}_{(2)} - \mathbf{f}_{(2)}.
		\end{align*}
		As the appearing Vandermonde matrices are invertible (since the $h_i$ are pairwise different), we can rewrite this as
		\begin{align}
			\label{eq_1}\mathbf{y} &= \mathbf{y}' - V(h_1, \ldots, h_d, h_{d+1})^{-1} \mathbf{f}_{(1)}\\
			\label{eq_2}\mathbf{y} &= \mathbf{y}' - V(h_1, \ldots, h_d, h_{d+2})^{-1} \mathbf{f}_{(2)}.
		\end{align}
		The inverse of a Vandermonde matrix can be explicitly computed and we find by considering the last coordinate of~\eqref{eq_1}
	\begin{align*}
		y_{d}' - y_{d} &= \sum_{j=1}^{d+1} \frac{1}{\prod \limits_{\substack{k = 1\\ k\neq j}}^{d+1} (h_j - h_k)} \cdot \frac{f_{h_j}}{m}.
	\end{align*}
	We use the representation $f_{h_j} = \varepsilon h_j^{d+1} (1+g_{h_j})$ to find
	\begin{align*}
		m(y_{d}' - y_{d}) &= \varepsilon \sum_{j=1}^{d+1} \frac{h_{j}^{d+1}(1+g_{h_j})}{\prod \limits_{\substack{k = 1\\ k\neq j}}^{d+1} (h_j - h_k)} .
	\end{align*}
	
	We first treat the main term, i.e. we disregard $g_{h_j}$.
	\begin{align*}
		\sum_{j=1}^{d+1} \frac{h_j^{d+1}}{\prod \limits_{\substack{k = 1\\ k\neq j}}^{d+1} (h_j - h_k)}
		&= \frac{\sum_{j=1}^{d+1} h_j^{d+1} (-1)^{d+1-j} \prod \limits_{\substack{i,k \in \{1, \ldots, d+1\}\setminus\{j\}\\ i<k}} (h_k - h_i)}{\prod \limits_{\substack{i,k \in \{1, \ldots, d+1\}\\ i<k}} (h_k - h_i)}.
	\end{align*}
	It is not difficult to see that this is a symmetric (rational) function in $h_1, \ldots, h_{d+1}$ and the degree of the numerator equals $d+1 + {d\choose 2}$, whereas the degree of the denominator equals ${d+1 \choose 2}$, i.e. the degree of the numerator is larger by one than the degree of the denominator.
	Moreover, we see that the numerator equals $0$ when $h_i = h_k$ for any $1\leq i < k \leq d+1$, i.e. $(h_k - h_i)$ divides the numerator.
	Thus, this expression is a symmetric polynomial of degree $1$, i.e. it is a multiple of $h_1 + \cdots + h_{d+1}$.
	Finally, we see that for fixed $h_1, \ldots, h_d$ and $h_{d+1} \to \infty$ this expression is asymptotically equivalent to $h_{d+1}$.
	This proves in total
	\begin{align*}
		\sum_{j=1}^{d+1} \frac{h_j^{d+1}}{\prod \limits_{\substack{k = 1\\ k\neq j}}^{d+1} (h_j - h_k)} 
		& = \sum_{j=1}^{d+1} h_j.
	\end{align*}
	Putting everything together, we find
	\begin{align*}
		\abs{m(y_d' - y_d) - \varepsilon \sum_{j=1}^{d+1} h_j}  &\leq \varepsilon \sum_{j=1}^{d+1} \frac{h_{j}^{d+1}\abs{g_{h_j}}}{\prod \limits_{\substack{k = 1\\ k\neq j}}^{d+1} \abs{h_j - h_k}}\\
			&\leq \varepsilon \sum_{j=1}^{d+1} \frac{H^{d+1} H^{- \kc}}{\prod \limits_{\substack{k = 1\\ k\neq j}}^{d+1} \abs{h_j - h_k}}\\
			&\leq \varepsilon (d+1) H^{d+1-\kc}.
	\end{align*}
	By replacing $h_{d+1}$ with $h_{d+2}$ - i.e. working with~\eqref{eq_2} instead of~\eqref{eq_1} - we find analogously
	\begin{align*}
		\abs{m(y_d' - y_d) - \varepsilon \rb{\sum_{j=1}^{d} h_j + h_{d+2}}}  &\leq \varepsilon (d+1) H^{d+1-\kc}.
	\end{align*}
	Via the triangle inequality, this shows that
	\begin{align*}
		\varepsilon \leq \abs{\varepsilon(h_{d+2} - h_{d+1})} \leq \varepsilon (d+1) H^{d+1-\kc} =\varepsilon (d+1) H^{-1},
	\end{align*}
	which is impossible for $H \geq d+2$.
		
	\textbf{Case 2}: The ``unshifted'' hyperplanes do not intersect in a single point.\\
	We consider two different intersection points of $d+1$ of the hyperplanes, i.e.
	\begin{align*}
		V(h_1, \ldots, h_d, h_{d+1}) \cdot \mathbf{y}^{(1)} &= \mathbf{s}_{(1)}\\
		V(h_1, \ldots, h_d, h_{d+2}) \cdot \mathbf{y}^{(2)} &= \mathbf{s}_{(2)}.
	\end{align*}
	Since $\mathbf{s}_{(i)} \in \Z/m$, we see that
	\begin{align*}
		\mathbf{y}^{(1)} &\in \frac{\Z}{m \cdot \abs{\det(V(h_1, \ldots, h_d, h_{d+1}))}}\\
		\mathbf{y}^{(2)} &\in \frac{\Z}{m \cdot \abs{\det(V(h_1, \ldots, h_d, h_{d+2}))}}.
	\end{align*}
	Thus, we have
	\begin{align*}
		y^{(1)} - y^{(2)} \in \frac{\Z}{m \cdot \abs{\prod_{1\leq i < j \leq d} (h_j - h_i) \cdot \prod_{1\leq i \leq d} (h_{d+1}-h_i)(h_{d+2}-h_i)}}.
	\end{align*}
	This shows in particular since $y^{(1)} \neq y^{(2)}$,
	\begin{align}
		\begin{split}\label{eq_4}
		\norm{y^{(1)} - y^{(2)}}_{\infty} &\geq \frac{1}{m \cdot \abs{\prod_{1\leq i < j \leq d} (h_j - h_i) \cdot \prod_{1\leq i \leq d} (h_{d+1}-h_i)(h_{d+2}-h_i)}}\\
			&\geq \frac{1}{m \cdot H^{\frac{d^2+3d}{2}}}.
		\end{split}
	\end{align}
	Moreover, we find similarly to the previous case
	\begin{align*}
		\mathbf{y} = \mathbf{y}^{(1)} - V(h_1, \ldots, h_d, h_{d+1})^{-1} \mathbf{f}_{(1)}.
	\end{align*}
	Thus, we can use again the representation $f_{h_j} = \varepsilon h_j^{d+1} (1+g_{h_j})$ to find for $0\leq i \leq d$,
	\begin{align*}
		\abs{y_i - y_i^{(1)}} &\leq \sum_{j=1}^{d+1} \frac{ e_{d-i}(\{h_1,\ldots, h_{d+1}\}\setminus\{h_j\})}{\prod_{k \in \{1,\ldots, d+1\}\setminus\{j\}}\abs{h_j - h_k}} \frac{f_{h_j}}{m}\\
			&\leq \sum_{j=1}^{d+1} {d \choose d-i} H^{d-i} \frac{\varepsilon h_j^{d+1} 2}{m}\\
			&\leq \frac{2\varepsilon}{m} (d+1) {d\choose d-i} H^{d-i+d+1}\\
			&\leq \frac{2}{m} (d+1){d\choose d-i} H^{2d+1-i-\kb}.
	\end{align*}
	Analogously, we find
	\begin{align*}
		\abs{y_i - y_i^{(2)}} \leq  \frac{2}{m} (d+1){d\choose d-i} H^{2d+1-i-\kb},
	\end{align*}
	which gives in total
	\begin{align}\label{eq_3}
		\abs{y_i^{(1)} - y_i^{(2)}} \leq \frac{4}{m} (d+1){d\choose d-i} H^{2d+1-i-\kb}.
	\end{align}
	In particular, for $H$ large enough in terms of $d$, we see that the right hand side of~\eqref{eq_3} is maximal for $i = 0$, which gives
	\begin{align*}
		\norm{y^{(1)} - y^{(2)}}_{\infty} \leq \frac{4}{m} (d+1) H^{2d+1-\kb},
	\end{align*}
	which gives a contradiction to~\eqref{eq_4} since $\kb = 2d+2+\frac{d^2+3d}{2}$.

	\end{proof}
	
	Thus, we are finally able to finish the proof of Theorem~\ref{Th2}.
	
	\begin{proof}[Proof of Theorem~\ref{Th2}]
		Throughout the proof of this theorem, we write 
		\[
			\ka = \max\rb{\frac{d^2+7d+4}{2(d+1-c)}, d+3}, \qquad \kb = \frac{d^2+7d+4}{2}, \qquad \kc = d+2.
		\] For any $n$ not satisfying Lemma~\ref{le_belongs_to_E} we see that $n \ll_c H^{\ka}$ and thus there are at most $O_c(H^{\ka})$ subwords of length $H$ in this range. For the remaining $n$, we can apply Lemma~\ref{le_belongs_to_E}, Proposition~\ref{pr_intersection}, and Proposition~\ref{pr_intersection_impossible} to see that any occurring subword of length $H$ corresponds to a non-empty region surrounded by limiting hyperplanes. As there are at most $O_c(m H^{d+2})$ limiting hyperplanes, there can be at most $O_c(m^{d+1} H^{d^2+3d+2})$ such regions (compare the case without error term), which finishes the proof.
	\end{proof}

\subsection{Lower bound for the complexity}

\begin{proposition}\label{lbcomplexity}
	Suppose that $c> 1$ is not an integer. Then, for any positive integer $H$, at least $H^2/4$ words of length $H$ occur in the sequence $(\floor{ n^c} \bmod m)$.
\end{proposition}
As mentioned above this lower bound is certainly not optimal. We expect that we should get (at least)
a lower bound of the form $c H^{\max(2,c(c-1)/2)}$, compare with \cite[Proposition~9.4]{AK2022}).
However, the lower bound $H^2/4$ is sufficient for our purposes since we only want to show that
the sequence $(\floor{ n^c} \bmod m)$ is not automatic.

\begin{proof}
	It was proved in~\cite{Deshouillers2019} that all words of length $\floor{c}+1$ occur in the sequence $(\floor{ n^c } \bmod m)$ and thus for any $H$, at least $m^2$ words occur; the proposition is thus proved for $H \le 2m$. From now on, we assume $H >2m$.
		
	We build sequences of length $H$ which are candidates for being subwords of $(\floor{ n^c } \bmod m)$. For $\ell \in \{m, m+1, \ldots, H-1\}$ and $i \in \{0, 1, \ldots, \ell-1\}$, we let
	\begin{equation}\notag
		u_{\ell, i} = \left(\floor{m\{i/\ell\}}, \ldots, \floor{m\{(i+h)/\ell\}}, \ldots, \floor{m\{(i+H-1)/\ell\}}\right).
	\end{equation} 
	Those sequences have the following properties: it easy to see that they have values in $\{0, 1, \ldots, m-1\}$, that they are non-decreasing, that they are not constant (since $(i+H-1)\ell -i/\ell \ge 1$) and that they are periodic with the smallest period $\ell$. Let us show that they are pairwise different: let us assume that $u_{\ell, i}=u_{k,j}$; they must have the same smallest period and thus $k=\ell$; let us assume, by contradiction, that $j>i$. If $u_{\ell, i}(0)\neq u_{k,j}(0)$, then $u_{\ell, i}\neq u_{k,j}$. If not, let us consider the smallest $h$ for which $u_{\ell, j}(h+1)\neq u_{\ell,j}(h)$ (it exists, since $u_{\ell, j}$ is not constant); since $i < j$, we have $u_{\ell, i}(h+1)= u_{\ell,i}(h)$, and thus $u_{\ell, i}\neq u_{\ell,j}$. The number of such sequences is
	\begin{equation}\notag
		\sum_{\ell = m}^{H-1} \ell = \frac{H(H-1)}2 - \frac {m(m-1)}2 \ge \frac{H^2}{2} - \frac{H}{2} - \frac{H}{4} \left( \frac{H}{2} - 1 \right) \ge \frac{H^2}{4}.
	\end{equation}
	
	For any $\ell \in \{m, m+1, \ldots, H-1\}$ and $i \in \{0, 1, \ldots, \ell-1\}$, there exist (infinitely many) integers $n$ such that
	\begin{equation}\notag
		\forall h \in [0, H-1] \colon \floor{ (n+h)^c } \bmod m = u_{\ell, i}(h).
	\end{equation}
	We let $d=\floor{c}$ and recall that 
	\begin{equation}\label{Taylorexp}
		(n+h)^c = \sum_{t = 0}^{d} \binom{c}{t} n^{c-t} h^t + O(h^{d+1}n^{c-d-1}).
	\end{equation}
	Let $\varepsilon$ be small enough in terms of $H$ and $c$. By an argument comparable to that of the proof of Theorem 1 of~\cite{Deshouillers2019} \footnote{Use of the multidimentional generalization of the Erd\H{o}s-Tur\'{a}n inequality due to Koksma and Sz\"{u}sz and of the van der Corput estimation of trigonometrical sums.} there exist infinitely many integers $n$ for which
	\begin{enumerate}
		\item $\left\{\frac{n^c}{m}\right\} \in \left[\frac{i}{\ell}+\varepsilon, \frac{i}{\ell} + \frac{1}{4m\ell d}-\varepsilon\right]$,
		\item $\left\{\frac{cn^{c-1}}{m}\right\} \in \left[\frac{1}{\ell}+\varepsilon, \frac{1}{\ell} + \frac{1}{4m\ell dH}-\varepsilon\right]$,
		\item $\forall t \in [2, d]\colon \left\{\frac{1}{m}\binom{c}{t}n^{c-t}\right\} \in \left[\varepsilon,  \frac{1}{4m\ell dH^d}-\varepsilon\right]$.
	\end{enumerate}
	Due to~\eqref{Taylorexp}, there exist an integer $K$ and a real $\theta$ in $(0,1)$ such that for any integer $h$ in $[0, H-1]$ one has
	\[
		\frac{(n+h)^c}{m} = K + \left\{\frac{i+h}{\ell}\right\} + \frac{\theta}{m\ell}.
	\]
	This implies that
	\[
		\floor{(n+h)^c} \bmod m  =\floor{m \left\{\frac{i+h}{\ell}\right\} } = u_{\ell, i}(h)
	\]
	and completes the proof of Proposition~\ref{lbcomplexity}.
\end{proof}

\section{Proof of Theorem~\ref{Th3}}
\label{sec_thm_3}

Although synchronizing automatic sequences $a(n)$ can be well approximated by periodic sequences, it is not obvious how upper bounds for the subword complexity of periodic sequences along $\lfloor n^c \rfloor$
can be transferred into upper bounds for the subword complexity of $a(\lfloor n^c \rfloor )$. 
It seems that we need more sophisticated considerations. In what follows we will first show that
the problem can be reduced to synchronizing automatic sequences $a(P(n))$ along polynomials $P(x)$ having integer coefficients.

\begin{proposition}\label{pr_transfer_deterministic}
	Let $a$ be a synchronizing $k$-automatic sequence, and $d\ge1$ an integer. Define
	\[
		\mathcal P_d\eqdef \bigl\{P\in \Z[X]:\deg P\le d,P(\N)\subseteq \N\bigr\}.
	\]

	If $a$ is uniformly deterministic along $\mathcal P_d$, that is,
\begin{align}\label{eqpr_transfer_deterministic}
	\lim_{H\to \infty}\frac 1H\log \#\bigl\{(a(P(n+\ell)))_{0\leq \ell<H}: n\geq 0, P\in \mathcal P_d\bigr\}= 0,
\end{align} then $a(\floor{n^c})$ is deterministic for all $c\in(0,d+1)$.
\end{proposition}

In a second step we will verify (\ref{eqpr_transfer_deterministic}) for $c \in \R\setminus \Z$ and $k\geq 2$ (see Proposition~\ref{pro3}).

%
%
%
%
%
%

For the proof of Proposition~\ref{pr_transfer_deterministic} we will need 
the following result which was attributed
to Hermann Weyl by Green and Tao \cite[Proposition~4.3]{Green2012b}.
\begin{proposition}\label{pr_approx}
	Suppose that $g: \Z \to \R$ is a polynomial of degree $d$ which we write as
	\begin{align*}
		g(n) = \beta_0 + n \beta_1 + \ldots + n^d \beta_d.
	\end{align*}
	Furthermore, let $\delta>0$ be sufficiently small. Then either the discrepancy of $(g(n) \bmod \Z)_{n\in \{1,\ldots,N\}}$ is smaller than $\delta$, or else there is an integer $1\leq \ell \ll \delta^{-O_d(1)}$, such that
	\begin{align*}
		\sup_{1\leq j \leq d} N^j \norm{\ell \beta_j} \ll \delta^{-O_d(1)}.
	\end{align*}
\end{proposition}

\begin{proof}
	This is basically~\cite[Proposition 4.3]{Green2012b} with a different representation of the polynomial and different notion of equidistribution.
	We sketch the proof how to transfer this result to our situation.
	We will need another representation of $g(n)$, namely
	\begin{align*}
		g(n) = \alpha_0 + {n\choose 1} \alpha_1 + \cdots + {n\choose d} \alpha_d.
	\end{align*}
	This allows us to apply~\cite[Proposition 4.3]{Green2012b} for $\delta_0 = \delta^{4}$, which can be assumed to be sufficiently small.
	Thus, we find that either $(g(n) \bmod \Z)_{n\in \{1,\ldots, N\}}$ is $\delta_0$-distributed, or there exists some $1\leq \ell_0 \ll \delta_0^{-O_d(1)}$ such that $\sup_{1\leq j \leq d} N^j \norm{\ell_0 \alpha_j} \ll \delta_0^{-O_d(1)}$.
	
	We start by discussing the first case. The Erd\H os--Tur\'an inequality implies that the discrepancy of $(g(n) \bmod \Z)_{n\in \{1,\ldots, N\}}$ is bounded, up to an absolute constant, by
	\begin{align*}
		\frac{1}{x} + \frac{1}{N} \sum_{y=1}^{x} \frac{1}{y} \abs{\sum_{n=1}^{N} \e(g(n)\cdot y)},
	\end{align*}
	uniformly for $x \in \N$.
	As $g(n)$ is $\delta_0$-distributed by assumption we can bound this by
	\begin{align*}
		\frac{1}{x} + \sum_{y=1}^{x} \frac{1}{y} \delta (1+y).
	\end{align*}
	We choose $x = \floor{\delta^{-1/2}}$ which shows after some short calculations that the discrepancy of $g(n)$ is bounded (up to an absolute constant) by $\sqrt{\delta_0}$, which can be bounded by $\delta_0^{1/4} = \delta$, when $\delta_0$ is small enough.
	Hence, the first case of~\cite[Proposition 4.3]{Green2012b} implies the first case of this proposition.
	In the second case, there exists some $1\leq \ell_0 \ll \delta_0^{O_d(1)} \ll \delta^{O_d(1)}$ such that
	\begin{align*}
		\sup_{1\leq j \leq d} N^j \norm{\ell \alpha_j} \ll \delta_0^{-O_d(1)} \ll \delta^{-O_d(1)},
	\end{align*}
	and it only remains to choose $\ell = d! \ell_0$ to finish the proof.
\end{proof}

Now we are ready to prove Proposition~\ref{pr_transfer_deterministic}.
\begin{proof}
	
	We recall that by \cite[Lemma 2.2]{Deshouillers2015} there exists some $\eta>0$ such that there are at most $k^{n(1-\eta)}$ words of length $n$ that are not synchronizing.
	
	We fix an arbitrarily small $\varepsilon > 0$ and choose $\lambda \in \N$ minimal such that $k^{-\eta \lambda} \leq \varepsilon$ and set $\delta = \frac{\varepsilon}{4 \cdot k^{\lambda}}$. (We note that $k^{\lambda} = O_{\varepsilon, a,k}(1)$ and $\delta = O_{\varepsilon, a, k}(1)$.)
	Throughout this proof, we let all implied constants depend on $\varepsilon, a, k$ and $c$. Instead of $\ll_{\varepsilon,a,k,c}$, we may therefore simply write $\ll$ without risk of ambiguity. Moreover, we allow the implied constants to change (finitely many times) throughout the proof and in particular from one line to the next line.
	
	We are interested in studying $(a(\floor{n^c}), a(\floor{(n+1)^c}), \ldots, a(\floor{(n+H)^c}))$, where $H$ tends to infinity.
	We recall that
	\begin{align*}
		(n+h)^c = P^{(n)}(h) + f_h^{(n)},
	\end{align*}
	where $f_h^{(n)} \ll n^{c-d-1} h^{d+1}$, for $n \geq 2 H$ and $P^{(n)}(h) = \beta_0^{(n)} + \beta_1^{(n)} h + \ldots + \beta_d^{(n)} h^d$.
	In particular $\abs{f_h^{(n)}} < \delta/2$ for $n \gg H^{O(1)} \delta^{-O(1)} \gg H^{O(1)}$.
	
	We apply Proposition~\ref{pr_approx} to $P^{(n)}/k^{\lambda}$ and distinguish the following two cases
	\begin{enumerate}
		\item \label{case_1} The discrepancy of $\rb{\{P^{(n)}(h)/k^{\lambda}\}}_{h\in [H]}$ is smaller than $\delta$.
		\item \label{case_2} There exists some $1\leq \ell \ll 1$ such that 
			\begin{align}\label{eq_cond_approx}
				\sup_{1\leq j \leq d} H^j \norm{\ell \beta_j/k^{\lambda}} \ll 1.
			\end{align}
	\end{enumerate}
	
	We start by considering case~\ref{case_1}. 
	A straightforward computation shows that the discrepancy of $\rb{\{(n+h)^c/k^{\lambda}\}}_{h\in [H]}$ is bounded by $4\delta$. (See for example~\cite{Deshouillers2015} for a very similar computation.)
	We compare $(a(\floor{(n+h)^c}))_{h\in [H]}$ with $(a(\floor{(n+h)^c} \bmod k^{\lambda}))_{h\in [H]}$. 
	As $a$ is synchronizing, these sequences can only differ when $\floor{(n+h)^c} \bmod k^{\lambda}$ is not synchronizing.
	Moreover, each residue class modulo $k^{\lambda}$ is hit at most $H/k^{\lambda} + 4 \delta$ times.
	Thus we find in total
	\begin{align*}
		&\abs{\{h \in [H]: a(\floor{(n+h)^c}) \neq a(\floor{(n+h)^c} \bmod k^{\lambda})\}}\\
		& \ \leq (H/k^{\lambda} + H 4 \delta) k^{\lambda(1-\eta)} \leq H k^{-\lambda \eta} + H 4\delta k^{\lambda} \leq 2 \varepsilon H.
	\end{align*}
	
	We recall that by Theorem~\ref{Th2}, there are at most $O\rb{k^{\lambda(d+1)} H^{O(1)}} = O(H^{O(1)})$ different subwords of $a(\floor{n^c} \bmod k^{\lambda})$ of length $H$. For each such subword we have to change at most $2 \varepsilon H$ letters to pass to $a(\floor{n^c})$. Thus, in this case we obtain at most $O\rb{H^{O(1)} \abs{\A}^{2 \varepsilon H}}$ different subwords of length $H$.
	
	Next we consider the case~\ref{case_2}, which implies, by multiplying~\eqref{eq_cond_approx} with $k^{\lambda}\ll 1$, the existence of some $1\leq \ell\ll 1$ such that
	\begin{align}\label{eq_approx_beta}
		\sup_{1\leq j \leq d} H^j \norm{\ell \beta_j^{(n)}} \ll 1.
	\end{align}
	
	This allows us to write (for $0\leq t < \ell$)
	\begin{align*}
		P^{(n)}(\ell h + t) &= \beta^{(n)}_0 + (\ell h + t) \beta_1^{(n)} + \ldots + (\ell h + t)^d \beta_d^{(n)}\\
			&= \gamma_0^{(n,t)} + h \gamma_1^{(n,t)} + \ldots + h^d \gamma_d^{(n,t)}
	\end{align*}
	Equation~\eqref{eq_approx_beta} implies that
	\begin{align*}
		\sup_{1\leq j \leq d} H^j \norm{\gamma_j^{(n,t)}} \ll 1.
	\end{align*}
	
	In particular, we can write
	\begin{align*}
		\gamma_j^{(n,t)} = z_j^{(n,t)} + s_j^{(n,t)},
	\end{align*}
	where $z_j^{(n,t)} \in \Z$ and $\abs{s_j^{(n,t)}} \ll H^{-j}$ for $j = 0, \ldots, d$.
	Putting everything together, we find
	\begin{align*}
		(n+\ell h +t)^c = Q^{(n,t)}(h) + r_h^{(n,t)}+ s_0^{(n,t)} ,
	\end{align*}
	where
	\begin{align*}
		Q^{(n,t)}(h) &= z_0^{(n,t)} + h z_1^{(n,t)} + \ldots + h^d z_d^{(n,t)}\\
		r_h^{(n,t)} &= f_{\ell' h + t}^{(n)} + h s_1^{(n,t)} + \ldots + h^d s_d^{(n,t)}.
	\end{align*}
	In particular, $Q^{(n,t)}$ is a polynomial of degree at most $d$ with integer coefficients and $r_h^{(n,t)}$ can be written as
	\begin{align*}
		r_h^{(n,t)} = (n+\ell h + t)^c - Q^{(n,t)}(h) - s_0^{(n,t)},
	\end{align*}
	that is, $r_h^{(n,t)}$ equals $(n+\ell' h + t)^c$ plus a polynomial of degree at most $d$. This shows that the $(d+1)$st derivative of $r_h^{(n,t)}$ is strictly positive for all $h\geq 0$. Thus, $r_h^{(n,t)}$ changes monotonicity at most $d$ times. 
	To see this, let us assume, in order to obtain a contradiction, that $r_h^{(n,t)}$ has at least $d+1$ extreme values (in particular the first derivative equals zero). 
	Thus, between each two consecutive extreme values there has to exist a point at which the second derivative equals zero, i.e. we have at least $d$ points at which the second derivative equals zero.
	Applying this reasoning repeatedly shows that there exist at least one point for which the $(d+1)$st derivative equals zero which gives a contradiction.\\
	Moreover, we see that 
	\begin{align*}
		\abs{r_h^{(n,t)}} &\ll \delta + 1\\
			&\ll 1,
	\end{align*}
	for $n \gg H^{O(1)}$.
	This shows in total, that we can decompose $[H]$ into $p \ll 1$ arithmetic progressions with step length $\ell \ll 1$ on which $\floor{(n+h)^c}$ equals a polynomial of degree at most $d$ in $h$ with integer coefficients. We call these arithmetic progression $AP_1, \ldots, AP_p$ and denote their lengths by $H_1, \ldots, H_p$ (in particular $H_1 + \ldots + H_p = H$).
	By assumption $a$ is deterministic along these arithmetic progressions.
	Thus, there exists a constant $c_{\varepsilon}$, such that the number of different words of length $H'$ is bounded by $c_{\varepsilon}\exp(\varepsilon H')$.
	 Moreover, there are at most ${H \choose p} \leq H^p$ different choices of the decomposition of $[H]$ into $p$ different arithmetic progressions of step length $\ell$. Thus, we have in total (summing over $\ell$ and $p$) the following upper bound for the number of different subwords of length $H$ in the second case:
	\begin{align*}
		\sum_{\ell \ll 1} &\sum_{p \ll 1}  H^p  \prod_{1\leq j \leq p} c_{\varepsilon} \cdot \exp(\varepsilon H_j)\\
		&= \sum_{\ell \ll 1} \sum_{p \ll 1}  (c_{\varepsilon} \cdot H)^p \cdot \exp(\varepsilon H)\\
		&\ll H^{O(1)} \exp(\varepsilon H).
	\end{align*}
	
	Combining case~\ref{case_1} and case~\ref{case_2} shows that the number of different subwords of length $H$ of $(a(\floor{n^c}))$ is bounded by $\ll H^{O(1)} \abs{\A}^{2 \varepsilon H}$, which finishes the proof as $\varepsilon>0$ was arbitrarily small.	
\end{proof}

As mentioned above it remains to check the condition (\ref{eqpr_transfer_deterministic}) (see Proposition~\ref{pro3}). 
For this purpose we need an equidistribution result of ``high'' digits of polynomials modulo $k^\lambda$ and the concept of the $k$-kernel of a sequence.

\subsection{Equidistribution of high digits of polynomials modulo $k^{\lambda}$}

In this subsection we deal with the distribution of ``high'' digits of polynomials modulo $k^{\lambda}$.
The main idea to use the Erd\H os--Tur\'an inequality, which reduces the problem to estimating exponential sums of the form $\sum_{n=1}^{k^{\lambda}} \e(P(n)\cdot y/k^{\lambda})$, where $P(n)$ is a polynomial with integer coefficients.

First, let us fix some notation. Throughout this section, we fix an arbitrary integer $k\geq 2$ and we write
\begin{align*}
	[w]_{\mu}^{\lambda} \eqdef \{ n \in \N: \exists 0 \leq n_0 < k^{\mu}, n_1 \in \N \text{ such that } n = n_0 + w \cdot k^{\mu} + n_1 k^{\lambda}\},
\end{align*}
where $\mu < \lambda, 0 \leq w < k^{\lambda-\mu}$. That is, $[w]_{\mu}^{\lambda}$ is the set of all integers $n$, such that the digits in base $k$ of $n$ between $\mu$ and $\lambda$ coincide with the digits of $w$ in base $k$.

\begin{proposition}\label{pr_high_digits}
	Let $d \in \N$. Then, for any polynomial $P(n) = z_d n^d + \cdots + z_0$ of degree $d$ with integer coefficients and any $q\in \N$ we have
	\begin{align*}
		D_{q}\rb{\frac{P(n)}{q}} &\ll_d \rb{\frac{q}{\gcd(z_1, \ldots, z_d,q)}}^{-1/(d+1)},
	\end{align*}
	where the implied constant only depends on $d$ (and not on $P$).
\end{proposition}
\begin{proof}
	We use the Erd\H os--Tur\'an inequality, which gives
\begin{align}\label{eq_disc_pol}
	D_{q}\rb{\frac{P(n)}{q}} \ll \frac{1}{x} + \frac{1}{q} \sum_{y=1}^{x} \frac{1}{y} \abs{\sum_{n=1}^{q} \e\rb{\frac{ P(n) \cdot y}{q}}},
\end{align}
for all $x\geq 1$, where the implied constant is absolute.

To estimate the innermost sum, one can use classical bounds due to Weyl or the following more precise bound, which is due to Nechaev~\cite{Necaev1975},
\begin{align*}
	\frac{1}{q}\abs{\sum_{n=1}^{q} \e\rb{\frac{Q(n)}{q}}} \ll_d q^{-1/d},
\end{align*}
where $Q(n) = a_d n^d + \ldots + a_0$ is a polynomial of degree $d$ with integer coefficients $a_0, \ldots, a_d$ such that $\gcd(a_1, \ldots, a_d, q) = 1$.
One of the most important features of this result is that it works uniformly for all polynomials $Q$ with the assumptions above.
It follows directly (by splitting the sum over $n$ into shorter sums)
\begin{align}\label{eq_estimate_pol}
	\frac{1}{q}\abs{\sum_{n=1}^{q} \e\rb{\frac{Q(n)}{q}}} \ll_d \rb{\frac{q}{\gcd(a_1, \ldots, a_d, q)}}^{-1/d},
\end{align}
where $Q(n) = a_d n^d + \cdots + a_0$ is a polynomial of degree $d$ with integer coefficients $a_0, \ldots, a_d$.	

We consider~\eqref{eq_disc_pol} and view now $P(n) \cdot y$ as a new polynomial $Q(n) = (z_d \cdot y) n^d + \ldots + (z_0 \cdot y)$ and see directly that $\gcd((z_1 \cdot y), \ldots, (z_d \cdot y), q) \leq y \cdot \gcd(z_1, \ldots, z_d,q)$.
Thus we find by using~\eqref{eq_estimate_pol}
\begin{align*}
	D_{q}\rb{\frac{P(n)}{q}} &\ll_d \frac{1}{x} + \sum_{y=1}^{x} \frac{1}{y} \rb{\frac{q}{y \cdot \gcd(z_1, \ldots, z_d,q)}}^{-1/d}\\
		&= \frac{1}{x} + \rb{\frac{q}{\gcd(z_1, \ldots, z_d,q)}}^{-1/d} \sum_{y=1}^{x} y^{-1+1/d}\\
		&\ll \frac{1}{x} + \rb{\frac{q}{\gcd(z_1, \ldots, z_d,q)}}^{-1/d} x^{1/d}.
\end{align*}

Choosing $x = \rb{\frac{q}{\gcd(z_1, \ldots, z_d,q)}}^{1/(d+1)} + O(1)$ gives
\begin{align*}
	D_{q}\rb{\frac{P(n)}{q}} &\ll_d \rb{\frac{q}{\gcd(z_1, \ldots, z_d,q)}}^{-1/(d+1)}.
\end{align*}
\end{proof}

\begin{corollary}\label{co_high_digits}
	Let $d\in \N_{>0}$ and $k\geq 2$. Let $p_0$ denote the smallest prime divisor of $k$ and $r \eqdef \log_k(p_0) \in (0,1]$.

	Let $P = z_d x^d + \ldots + z_0 \in \Z_d[x]$ be a polynomial such that $k\nmid z_j$ for some $1\leq j \leq d$.
We have the estimate
	\begin{align*}
		\#\{n<k^{\lambda}: P(n) \in [w]_{\floor{\lambda(1-r/(d+1))}}^{\lambda}\} \ll_{d,k} k^{\lambda (1-r/(d+1))},
	\end{align*}
	where the implied constant only depends on $d$ and $k$.
\end{corollary}
\begin{proof}
	We note that by our assumptions, $\gcd(z_1, \ldots, z_d, k^{\lambda}) \leq \frac{k^{\lambda}}{p_0^{\lambda-1}}$, so that
	\begin{align*}
		D_{k^{\lambda}}\rb{\frac{P(n)}{k^{\lambda}}} \ll_k k^{-\lambda r/(d+1)}.
	\end{align*}
	Moreover, we recall that $P(n) \in [w]_{\floor{\lambda(1-r/(d+1))}}^{\lambda}$ if and only if
	\begin{align*}
		\cb{\frac{P(n)}{k^{\lambda}}} \in \left[\frac{w}{k^{\lambda - \floor{\lambda(1-r/(d+1))}}}, \frac{w+1}{k^{\lambda - \floor{\lambda (1-r/(d+1))}}}\right),
	\end{align*}
	where the right hand side is an interval of length at most $k^{-\lambda r/(d+1)} = p_0^{-\lambda/(d+1)}$.
	The result now follows directly from the definition of the discrepancy.
\end{proof}

\subsection{Subword complexity of automatic sequences along polynomials}

We recall the definition of the $k$-kernel of a sequence.
\begin{definition}
	The $k$-kernel of a sequence $a(n)$ is the following set of subsequences:
	\begin{align*}
		\Ker_k(a) \eqdef \{(a(nk^{\lambda} + r)_{n\in \N}: \lambda \in \N, 0\leq r < k^{\lambda}\}.
	\end{align*}
\end{definition}
The $k$-kernel of a sequence is a priori infinite and can be used to determine whether it is automatic.
\begin{theorem}[\cite{Cobham1972}]
	A sequence $a(n)$ is $k$-automatic if and only if its $k$-kernel $\Ker_k(a)$ is finite.
\end{theorem}


It is straightforward to prove the following result.
\begin{lemma}\label{le_synchronizing}
	Let $a(n)$ be a synchronizing $k$-automatic sequence. Then all the sequences $b_i \in \Ker_k(a)$ are also synchronizing $k$-automatic sequences.
	Furthermore, $\bfw$ is a synchronizing word for $a$, then it is also a synchronizing word for all $b_i \in \Ker_k(a)$.
In other words, we can choose the same exponent $\eta$ for all the sequences $b_i\in \Ker_k(a)$ as for $a$ (confer~\cite[Lemma 2.2]{Deshouillers2015}).
\end{lemma}

The next proposition (together with Proposition~\ref{pr_transfer_deterministic}) proves Theorem~\ref{Th3}.

\begin{proposition}\label{pro3}
	Let $a$ be a synchronizing $k$-automatic sequence, where $k\geq 2$ and let $d \in \N$.
Then 
\begin{align*}
\lim_{H\to \infty} \frac 1H\log \#\bigl\{(a(P(n+\ell)))_{0\leq \ell<H}: n\geq 0, P\in \mathcal P_d\bigr\} = 0.
\end{align*}
\end{proposition}
\begin{remark}
	It is relatively straight-forward to study the subword complexity of a synchronizing automatic sequence along a single polynomial (which can be done using Hensel's Lemma), but the real difficulty comes from considering all polynomials of a given degree at the same time.
\end{remark}
\begin{proof}
We consider integer polynomials $P(n)$ and aim to study the words
\[\bigl(a(P(n)), a(P(n+1)), \cdots, a(P(n+H-1))\bigr).\]
	 By considering the Taylor expansion of $P$, we can write
	 \begin{align*}
	 	P(n + h) &= P(n) + h \cdot P'(n) + \ldots + h^{d} \frac{P^{(d)}(n)}{d!}\\
	 		&= Q^{(n)}(h),
	 \end{align*}
	 where $Q^{(n)} \in \Z_d[x]$. That is, we can assume without loss of generality that $n = 0$.
	 
	 We write $P(h) = z_d h^d + \cdots + z_1 h + z_0$, where $z_j \in \Z$ for $j = 0, \ldots, d$.	 
	The case where $z_d = \cdots = z_1 = 0$ is trivial, so that we only focus on the case $z_j \neq 0$ for some $j \in \{1, \ldots, d\}$.
	We denote by 
	\begin{align*}
		\lambda_0 = \min_{j\in \{1, \ldots, d\}} \max \{k\in \N: \lambda^k|z_j\}.
	\end{align*}
	 Thus, we find for $j = 1, \ldots, d$ some $z_j' \in \Z$ such that $z_j = k^{\lambda_0} z_j'$.
	 Our choice of $\lambda_0$ guarantees that there exists $j \in \{1,\ldots, d\}$ such that $k$ does not divide $z'_j$.
	Moreover, we rewrite $z_0 = r + z_0' k^{\lambda_0}$, where $0\leq r < k^{\lambda_0}$ and $z_0' \in \N$.
	Thus, we have
	\begin{align*}
		P(h) = k^{\lambda_0} \cdot \rb{z_d' h^d + \cdots + z_1' h + z_0'} + r.
	\end{align*}
	Thus, we can write
	\begin{align*}
		a(P(h)) = b_i(z_d' h^d + \cdots + z_1' h + z_0'),
	\end{align*}
	where $b_i (n) = a(nk^{\lambda_0} + r)$ belongs to the (finite) $k$-kernel of $a$.  In particular, $b_i$ is again synchronizing and we can use the same synchronizing exponent $\eta$ (see Lemma~\ref{le_synchronizing}).
	We fix an arbitrary $\varepsilon > 0$ and choose $\lambda \in \N$ minimal such that $k^{-\lambda \eta} \leq \varepsilon$ and $k^{-\lambda \cdot r \cdot (1-\eta)/(d+1)} \leq \varepsilon$. 
	
	Let us compare $b_i(P(h))$ and $b_i( P(h) \bmod k^{\lambda})$.
	These two differ only when $z_d' h^d + \cdots + z_0' \bmod k^{\lambda}$ is not synchronizing, which implies that its digits between positions $\floor{(1-r/(d+1)) \lambda}$ and $\lambda$ are also not synchronizing.
	To be more precise, we will assume that $P(h) \bmod k^{\lambda}$ is not synchronizing. If $w < k^{\lambda - \floor{(1-r/(d+1))\lambda}}$ is chosen such that $P(h) \in [w]_{\floor{(1-r/(d+1))\lambda}}^{\lambda}$, then $w$ is also not synchronizing.
	However, the number of such $w <k^{\lambda - \floor{(1-r/(d+1))\lambda}}\asymp k^{\lambda r/(d+1)} = p_0^{\lambda/(d+1)}$ is bounded from above by $O(k^{\lambda r/(d+1) \cdot (1-\eta)})$, where the implied constant only depends on $a$ and $k$.
	We find by Corollary~\ref{co_high_digits}
	\begin{align*}
		\#\{0 \leq h < k^{\lambda}: P(h) \in [w]_{\floor{(1-r/(d+1))\lambda}}^{\lambda} \} \ll k^{\lambda r/(d+1)} = p_0^{\lambda/(d+1)},
	\end{align*}
	where the implied constant only depends on $d$ and $k$.
	
	This shows
	\begin{align*}
		\#\{0 \leq h < k^{\lambda}&: \rb{P(h) \bmod k^{\lambda}} \text{ is not synchronizing}\} \ll_{a,k,d} \\
			&\ll_{a,k,d} k^{\lambda r/(d+1) \cdot (1-\eta)} k^{\lambda (1-r/(d+1))} = k^{\lambda (1- r\cdot (1-\eta)/(d+1))}\leq \varepsilon k^{\lambda}.
	\end{align*}
	Thus, we have proved that the number of integers $h <k^{\lambda}$ such that $b_i(P(h))$ and $b_i(P(h) \bmod k^{\lambda})$ may differ is bounded from above by $c_{a,k,d}\rb{\varepsilon k^{\lambda}}$, where $c_{a,k,d}$ is a constant only depending on $a, k$ and $d$.
	Moreover, since we only used properties of $P(h) \bmod k^{\lambda}$, this observation holds for any interval of length $k^{\lambda}$. Therefore, by subdividing $[0, H-1]$ into intervals of size at most $k^{\lambda}$, we see that 
\[b_i(P(h)) \neq b_i(P(h) \bmod k^{\lambda})\]
for at most
	\begin{align*}
		\floor{H/k^{\lambda} + 1} c_{a,k,d} \varepsilon k^{\lambda} \leq c_{a,k,d} (\varepsilon H + k^{\lambda})
	\end{align*}
choices of $h \in \{0, \ldots, H-1\}$.
	We recall that $b_i(P(h) \bmod k^{\lambda})$ only depends on $i$ and the coefficients of $P$ modulo $k^{\lambda}$ which shows that there are at most $\abs{\Ker_p(a)} \cdot (k^{\lambda})^{d+1}$ different such sequences. 
	Moreover, each such sequence is a periodic sequence, which means that its subword complexity is bounded by a constant.
	By a similar reasoning as in the proof of Proposition~\ref{pr_transfer_deterministic}, we can bound the number of different subwords of length $H$ that appear in any sequence of the form $a(P(n))$ by
	\begin{align*}
		O\rb{\abs{\Ker_p(a)} \cdot k^{\lambda (d+1)} \abs{\A}^{c_{a,k,d}(\varepsilon H + k^{\lambda})}} = O_{\varepsilon, a, k,d} \rb{\abs{\A}^{c_{a,k,d} \varepsilon H}}.
	\end{align*}
This finishes the proof as $c_{a,k,d}$ is independent of $\varepsilon$ and $\varepsilon>0$ can be chosen arbitrarily small.
\end{proof}

\bigskip\noindent
{\bf Acknowledgement.} 
The work is supported by the Austrian-French project ``Arithmetic Randomness'' between FWF and ANR (grant numbers I4945-N and ANR-20-CE91-0006). The authors are grateful to Jakub Koniecny for the reference~\cite{AK2022} and for pointing out that our results can be generalized to Hardy fields functions $g(x)$ of growth $g(x) \sim x^{c+o(1)}$.

\bibliographystyle{abbrv}
\bibliography{bibliography_general}

\end{document}